\documentclass[12pt]{article}
\usepackage{amsmath,amsfonts,amssymb,theorem,graphicx, verbatim}
\usepackage{hyperref}
\usepackage{pb-diagram}

\newcommand{\CC}{\mathbb C}

\newcommand{\AAA}{\mathbb A}
\newcommand{\PP}{\mathbb P}
\newcommand{\ZZ}{\mathbb Z}

\newcommand{\QQ}{\mathbb Q}

\newcommand{\lam}{\lambda}

\newcommand{\Lam}{\Lambda}
\newcommand{\al}{\alpha}
\newcommand{\be}{\beta}
\newcommand{\gog}{\mathfrak g}

\newcommand{\om}{\omega}

\newcommand{\ep}{\varepsilon}
\newcommand{\Si}{\Sigma}

\newcommand{\Oh}{\mathcal O}

\newcommand{\sL}{\mathcal L}

\newcommand{\into}{\hookrightarrow}
\newcommand{\onto}{\twoheadrightarrow}
\newcommand{\ra}{\rightarrow}

\newcommand{\color}[6]{}

\theoremstyle{change}   
\theorembodyfont{\rmfamily}
 \newtheorem{theorem}[subsection]{Theorem}

 \newtheorem{prop}[subsection]{Proposition}

\newtheorem{example}[subsection]{Example}
\newtheorem{rmk}[subsection]{Remark}
{
\theorembodyfont{\rmfamily}

 \newtheorem{nothing}[subsection]{}
}
 
 \newenvironment{proof}{\paragraph{Proof}}{\par\medskip}

\theorembodyfont{\rmfamily}

\DeclareMathOperator{\SL}{SL}
\DeclareMathOperator{\GL}{GL}
\DeclareMathOperator{\hcf}{hcf}

\DeclareMathOperator{\Gr}{Gr }

\DeclareMathOperator{\Proj}{Proj}
\DeclareMathOperator{\Spec}{Spec}

\numberwithin{equation}{section}
\usepackage[normalem]{ulem}
\begin{document}

\title{Constructing projective varieties \\ in weighted flag varieties}
\date{}
\author{Muhammad Imran Qureshi and Bal{\'a}zs Szendr{\H{o}}i}

\maketitle
\begin{abstract}
We compute the Hilbert series of general
weighted flag varieties and discuss 
a computer-aided method to determine their defining equations. 
We apply our results to weighted flag varieties coming from the Lie groups 
of type $G_2$ and $\GL(6)$, to construct some families of polarised  
projective varieties in codimensions~$8$ and~$6$, respectively.  
\end{abstract}

\section{Introduction} The aim of this article is twofold. Firstly, we prove 
a formula for the Hilbert series of general weighted flag varieties of
Grojnowski and Corti--Reid~\cite{wg} and discuss 
a computer-aided method to determine their defining equations. 
The Hilbert series is expressed in terms of Lie-theoretic data 
associated to the weighted flag variety. The equations are determined 
using an explicit construction of highest weight modules over Lie algebras. 
Secondly, we use these results to construct families of projective varieties
in higher codimension by taking quasi-linear sections of weighted flag 
varieties. As examples, we exhibit in this article some new families of 
Calabi--Yau threefolds in codimensions 6 and 8. A more complete list of 
families arising from this construction will be given elsewhere~\cite{Q}. 

We are interested in polarized varieties $(X,D)$, projective varieties $X$ 
polarised by a \(\QQ\)-ample Weil divisor \(D\) such that some integer 
multiple of \(D\) is Cartier. 
Such a polarized variety $(X,D)$ gives rise to a finitely generated graded ring
\begin{displaymath}
R(X,D)=\bigoplus_{n \geq 0}H^0(X,nD).
\end{displaymath} 
A surjection 
\begin{displaymath}
k[x_0, \cdots , x_n]\onto R(X,D)
\end{displaymath}
from a free graded ring \(k[x_0, . . . , x_n]\) generated by variables \(x_i\) of weights \(w_i\) 
gives an embedding of \((X, D)\) into weighted projective space
\[i : X = \Proj R(X,D) \into \PP[w_0, \cdots , w_n]\]
with the divisorial sheaf \(\Oh_X(D)\) being isomorphic to \(\Oh_X(1) = i^*\Oh_{\PP}(1)\). 

We are looking for examples where this embedding is of relatively low codimension.
We work in the framework reviewed in~\cite{ABR}, where the classical case of codimension at 
most 3 is discussed. See~\cite{altinok} for codimension 4 examples, and Corti and Reid \cite{wg} 
for examples in codimension 5. 

We construct examples in codimensions 6 and 8 arising as quasilinear sections of particular 
weighted flag varieties $(w\Sigma, \Oh_{w\Sigma}(1))$ embedded in weighted projective space 
by their natural Pl\"ucker-type embeddings. We look for candidate examples by 
computing the Hilbert series of a weighted flag variety of a given type. 
To understand the singularities of quasilinear sections, we need to know
the defining equations of the flag variety. Defining ideals of flag varieties are described 
in~\cite[Sec 1]{rudakov} in Lie-theoretic terms; 
we work out the equations by a GAP4 code, using an explicit construction of
representations which appears in~\cite{degraaf}.
In theory, this approach allows one to search for examples in arbitrarily high 
fixed codimension.

After fixing notation in Section~\ref{notationsec}, in Section \ref{hssec} we prove 
a formula for the Hilbert series $P_{w\Sigma}$ of a weighted flag variety \(w\Si \). 
We also describe a method of determining the defining equations, and explain how to 
construct families of polarized varieties such as Calabi--Yau threefolds as quasi-linear 
sections of \(w\Si\). In Section \ref{g2sec}, we study weighted flag varieties associated 
to the Lie group of type \(G_2\), leading to the codimension eight varieties. The case of
weighted homogeneous \(w\)Gr(2,6) in codimension six is discussed in Section \ref{g26sec}.
Complete lists of families that can be constructed in these ambient varieties, as well as
other examples of higher codimension, will appear elsewhere \cite{Q}. 

\subsection*{Acknowledgements} We wish to thank Willem De Graaf for 
providing us with the GAP4\ code to compute the defining equations of 
flag varieties and Richard Williamson for further help with GAP. 
We also thank Gavin Brown and Miles Reid for helpful discussions. 
The first author has been supported by a grant from the Higher Education 
Commission (HEC) of Pakistan.

\section{Definitions and conventions}\label{notationsec}

\subsection{Algebraic geometry} We work over a field \(\CC\) of complex numbers. A polarised 
variety is a pair \((X,D)\), where~\(X\) is a normal projective algebraic variety and~\(D\) is an 
ample $\QQ$-Cartier Weil divisor on~X. 

For positive weights $w_0,w_1,\cdots,w_n $, we use the standard notation
$\PP[w_0,w_1,\cdots,w_n]$ for weighted projective space; sometimes we will write $w\PP$
if the weights are clear. The weighted projective space \(\PP[ w_i]\) is called well-formed, 
if no \(n-1\) of \(w_0,\cdots,w_n\) have a common factor. 
The point $P_i\in w\PP$ with coordinates $[0,\dots,0,1,0,\dots, 0]$, 
where the 1 is in the $i$-th position, is a vertex, the 1-dimensional\index{toric!stratum} 
stratum $P_i P_j$ an edge, etc. Defining $h_{i,j,\dots}=\hcf(w_i,w_j,\dots)$, the vertex $P_i$ 
of a well-formed weighted projective space is 
a singularity of type $$\frac{1}{w_i}(w_0,\dots, \widehat {w_i},\dots,w_n).$$ 
Each generic point $P$ of the edge $P_i P_j$ has an analytic 
neighbourhood $P\in U$ which is analytically isomorphic to $(0,Q)\in \AAA^1 \times X$, 
where $Q\in X$ is a singularity of type $$\frac{1}{h_{i,j}}(w_0,\dots,\widehat{w_i},\dots,
\widehat{w_j},\dots,w_n).$$ Similar results hold for higher dimensional strata.

A projective subvariety \(X\subset \PP^n[w_i]\) of codimension \(c\)  is called well-formed, 
if \(\PP^n[w_i] \) is well-formed and \(X\) does not contain a codimension \(c+1\) singular 
stratum of \( \PP[w_i]\). 
The subvariety \(X\subset \PP[w_i]\) is quasi-smooth, if the affine cone 
\(\widetilde X \subset \AAA^{n+1}\) of \(X\) is smooth outside its vertex \(\uline{0}\). 
If \(X\) is quasi-smooth, then it will only have quotient singularities induced by the  
singularities of \(\PP[w_i]\). The subvariety \(X\subset \PP[w_i]\) will always be 
assumed to be polarized by the restriction of the tautological ample divisor 
$\Oh_{\PP}(1)$. 

The Hilbert series of a polarised projective variety \((X,D)\) is 
\[ P_{(X,D)}(t)=\sum_{n \geq 0}\dim H^0(X,nD) t^{n}.
\]
We will sometimes write $P_X(t)$ if no confusion can arise. Appropriate Riemann--Roch 
formulas, together with vanishing, can be used to compute $h^0(X,nD)=\dim H^0(X, nD)$ 
in favourable cases. 

A polarised Calabi--Yau threefold \((X,D)\) is a Gorenstein, normal, projective 
three dimensional algebraic variety with \(K_{X} \sim 0\)  and 
\(H^{1}(X,\Oh_X)=H^2(X,\Oh_X)=0\). We allow \((X,D)\) to have at worst canonical quotient 
singularities, consisting of points and curves on \(X\).
If \(C\) is a curve of singularities of such an \((X,D)\), take a generic surface 
\(S\subset X\) such that \(C\) and \(S\) intersect transversely in finite number of points. 
If every point in the intersection is a singular point of type \(\frac{1}{r}(1,-1)\) 
on polarised surface \((S,D\vert_S)\), then \(C \in X\) is called a curve of singularities 
of type \(\frac{1}{r}(1,-1)\), known  as an \(A_{r-1}\) curve. Each \(A_{r-1}\) curve contains 
a finite number of dissident points, where the singularity is worse. If an \(A_{r-1}\) curve 
\(C\) contains the dissident points \(\lbrace  P_{\eta}: \eta \in \Lam \rbrace \) of type 
\(\frac{1}{r \tau_{ \eta}}(w_{1\eta},w_{2 \eta},w_{3 \eta})\), then we define the index of 
\(C\) to be \begin{displaymath}
\tau_{C}=\text{lcm}_{\eta \in \lam}  \lbrace{ \tau_{\eta}\rbrace}.
\end{displaymath}
For an \(A_{r-1}\) curve \(C\) with no dissident points, \(\tau_{C}=1\). 
The contribution of $C$ to the Riemann--Roch formula \cite{anita} 
is determined by \(\tau_{C}\) as well as a further invariant \(N_C\) depending on 
the normal bundle of $C$ in $X$; see \cite[p.16]{anita} for a detailed description.

%Let \(P\) be the singular point of type \(\dfrac{1}{r}(w_1,w_2,w_3)\), then the  integers \(r,w_{1},w_2 \text{ and }w_3\) determine the type of \(P.\) If \(r\) is coprime to each of  the \(w_i\) then \(P\) is an isolated singular point. If \(\gcd (r,w_i\)) is non trivial for one of \(w_i\),s then \(P\) is called a  dissident singular point which lies on a dissident curve. If \(\gcd(r,w_i)>1  \), for two of the \(w_i\),s or for all  three of them, then \(P\) is the dissident singular point lying on the intersection of two or three dissident curves respectively. 

\subsection{Representation theory}
Let  \(G \) to be a reductive Lie group  with corresponding Lie algebra  $\mathfrak g$. 
We denote by \(B\) a Borel subgroup of \(G\), by \(P\)  a parabolic subgroup of \(G\), 
and by \(T\) a maximal torus, such that \(T\subset B \subset P \subset G\). 
Let $\Lam_W={\rm Hom}(T, \CC^*)$ be the weight lattice. 

We recall that any irreducible representation \(V\) of  \(G\) has a decomposition 
\[V=\bigoplus_{\al\in\Lambda_W} V_{\al}\]
into eigen\-spaces under the action of the maximal torus \(T\);
the \(\al\)'s with nonzero $V_\alpha$ are the weights of the representation \(V\). 
The nonzero weights of the adjoint representation are called the roots. 
The root lattice is the sublattice \(\Lam_R\) of  $\Lam_W$
generated by the roots. Their dual lattices of 
one-parameter subgroups are denoted by \(\Lam_R^* \) and \(\Lam^*_W\) respectively; 
we have a perfect pairing \(<,>:\Lam_W\times \Lam_W^*\ra\ZZ. \) 

We denote the set of roots by \(\nabla \) and fix a decomposition into positive 
and negative roots, \(\nabla_+\cup \nabla_-\). The set of 
positive roots that can not be written as a sum of two other positive 
roots form the set of simple roots  \(\nabla_0\) of $G$. Any root \(\al \in \nabla\) 
determines an involution of the maximal torus T, and hence a reflection \(s_{\al}\).  
These reflections generate the Weyl group \(W\) of \(G\). The reflections corresponding 
to the simple roots are called simple reflections.

There is a partial order defined on the weights of any representation by \(\al \leq \be\) if and only if \(\be - \al\) is non-negative linear combination of simple roots. A weight \(\lam \) of a representation is called the highest or dominant  weight of the representation  if no other weight is greater than \(\lam\) with respect to this partial order. For semisimple $G$, there is a set of weights \(\om_1,\cdots, \om_n \) in \(\Lam_W\), called the fundamental weights, defined by the property that any dominant weight can be written uniquely as a non-negative integral linear combination of the fundamental weights.

The quotient of the Lie group \(G\) by a Borel subgroup, \(\Si=G/B\) is 
the complete flag variety. Quotients of the form \(\Si=G/P,\) for $P\subset G$ 
a parabolic, are called (generalized) flag varieties or projective homogeneous spaces
in the literature; we will use the words flag variety and homogeneous variety 
interchangeably.

Let $p\in \PP V_{\lam}$ be the highest weight vector in the representation $V_\lam$ of $G$
with highest weight $\lam\in  \Lam_W$. Let $P_\lambda$ be the parabolic subgroup corresponding  
to those elements in \(\nabla_0\), which are orthogonal to the weight \(\lam\) 
in the weight lattice. Then the unique closed orbit $G.p$ of the action of $G$ on  
$\PP V_{\lam}$ realizes the flag variety $ \Sigma = G/P_{\lambda}$ as a projective
subvariety of $\PP V_{\lam}$. All flag varieties $G/P$ can be realised as projective 
varieties in this way. 

As an example, recall that if \(G=\SL(n,\CC)\), then \(B\) is a subgroup of upper 
triangular matrices, and $G/B$ is the variety of all flags \begin{displaymath}
\Si=\SL(n,\CC)/B=\lbrace{ 0 \subset V_1 \subset V_2 \subset \cdots \subset V_n=V \rbrace } 
\end{displaymath} of subspaces, with \(\dim V_k=k.\) The group \(G\) acts transitively on \(\Si\) and \(B\) is the stabilizer of a fixed flag. If \(P\) is a  parabolic subgroup of \(SL(n,\CC)\) containing \(B\) then  the quotient \[\Si= \SL(n,\CC)/P=\lbrace 0 \subset V_i \subset V_m \subset \cdots \subset V_n \rbrace\] is a (generalised) flag variety, with \(P\) being the stabilizer of a fixed partial flag.

\section{Weighted flag varieties}\label{hssec}
\subsection{Hilbert series formula}
We start by  recalling the notion of weighted flag variety due to Grojnowski and Corti--Reid \cite{wg}. 
Fix a reductive Lie group $G$ as above, and a highest weight \(\lam \in \Lam_W \).  
Let $\Sigma=G/P_\lam$ be the corresponding flag variety as above.
Choose \(\mu \in \Lam_W^* \) and an integer \(u \in\ZZ\) such that \begin{equation}
<w\lam,\mu>+u >0
\label {weights} 
\end{equation}
for all  elements \(w\) of the Weyl group. 
Consider the affine cone $ \widetilde{\Sigma} \subset  \widetilde{V_{\lambda}}$ of the 
embedding $\Sigma \hookrightarrow \PP V_{\lambda}$,
and divide it by the $ \CC^*$-action on $V_{\lambda}-\{0\}$ given by
\[ 
(\varepsilon \in \CC^*) \mapsto ( v \mapsto \varepsilon^u(\mu(\varepsilon)\circ v)).
\]
Denote this variety by \(w\Si(\mu,u)\) or $w\Sigma$ if no confusion can arise. 
The inequality (\(\ref{weights}\)) ensures that all the weights 
on \(w\Si(\mu,u)\subset w\PP V_{\lam} \) are positive, so the quotient indeed exists.

\begin{theorem} The Hilbert series of the weighted flag variety 
$\left(w\Sigma(\mu,u),D\right)$ has the following closed form.
\begin{equation}
P_{w\Si}(t)=\dfrac{\sum_{w\in W}(-1)^w \dfrac{t^{<w\rho, \mu>}}{(1- t^{<w\lambda,\mu>+u})}}                                                                     {\sum_{w\in W}(-1)^w t^{<w\rho, \mu>} }.
\label{whhs}
\end{equation}
Here \(\rho\) is the Weyl vector, half the sum of the positive roots of \(G\), and  $(-1)^w=1  \; \mbox{or} -1$ depending on whether $w$ consists of an even or odd number of simple reflections in the Weyl group $W$.
\end{theorem}
\begin{proof} Any weight \(\lam\) of \(G\)-representation \(V_{\lam}\) gives rise to a line bundle \(\sL_{\lam}\) on the straight flag variety \(\Si\). If \(\lam\) is dominant, then by the Borel--Bott--Weil theorem, the space of holomorphic sections of the line bundle \(\sL_{\lam}\) is the irreducible representation with highest weight \(\lam\).

Let \(D=\Oh_{w\Si}(1)\) under the embedding \(w\Si \subset w\PP V_{\lam}\), then we have a graded ring
\begin{equation}
R= \bigoplus_{n \geq 0} H^0(w \Sigma, \mathcal{O}_{w\Sigma}(nD)).
\label{wgr}
\end{equation} 
By construction, \( \Spec(R)=\widetilde {w\Si}\), the affine cone over the weighted flag variety \(w\Si\). However, this affine cone is the same as the  straight affine cone \(\widetilde{\Si}\), so by the Borel--Bott--Weil theorem, we have
\begin{equation}
R= \bigoplus_{m \geq 0} V_{m\lambda}. 
\label{sgr}
\end{equation} 
The grading on $ R  $ given by (\ref{wgr}) is different from the grading given by (\ref{sgr}) because of the weights. The Hilbert series of the graded ring  (\ref{wgr}) is given by 
\begin{equation}
P_{w\Si}(t)= \sum_{n \geq 0} h^0(w \Sigma, \mathcal{O}_{w\Sigma}(nD)) t^n.
\label{whs}
\end{equation}
By the Weyl character formula, 
\begin{equation}
{\rm Char}( V_{\lambda}) = \dfrac{ \sum_{w\in W}(-1)^w t^{w(\lambda +\rho)}}{\sum_{w\in W}(-1)^w t^{w(\rho)}} .
\label{wcf}
\end{equation}
Therefore by (\ref{whs}) and (\ref{wcf}), the Hilbert series of the weighted flag variety with weights  $<w\lambda, \mu > + u  $, which are positive by (\ref{weights}), is given by the infinite sum
$$ P_{w\Si}(t)= \sum_{n \geq 0}\dfrac{\sum_{w\in W}(-1)^w t^{<w(n\lam+\rho),\mu>+nu}}{\sum_{w \in W}(-1)^w t^{<w\rho,\mu>}} .$$
After rearrangement of terms we get the claimed formula~\eqref{whhs}.
\end{proof}

\begin{rmk}
Since~\eqref{whhs} involves summing over the Weyl group, it is best to use a computer 
algebra system for explicit computations. We used GAP, SAGE and Mathematica in combination 
to work out the Hilbert series formulas for the flag varieties appearing below.
As immediate examples, note that the formulas appearing in \cite[p.8 and p.23]{wg} 
can be derived from~\eqref{whhs} with little effort.
\end{rmk}

\begin{rmk} Note that by the Weyl denominator identity, the denominator of our expression for 
the Hilbert series has two equivalent forms 
\begin{equation}
\sum_{w\in W}(-1)^w t^{<w\rho,\mu>}=t^{<\rho,\mu>} \prod_{\al \in \nabla_+}(1-t^{<-\al,\mu>}).
\label{hdidentity}
\end{equation}
Note also that by standard 
Hilbert--Serre theorem \cite[Theorem 11.1]{atiyah}, 
the Hilbert series of the weighted flag variety has a reduced expression
\begin{equation}
P_{w\Si}(t)=\dfrac{N(t)}{ \prod_{w_i \in \Lam_{W}}(1-t^{<w_i,\mu>+u})}.
\label{reducedhs}
\end{equation}
Here the product is over all the weights of the representation \(V_{\lam}\);  
the polynomial \(N(t)\) is called the Hilbert numerator of the Hilbert series.
There is a lot of cancellation happening when going from~\eqref{whhs} to~\eqref{reducedhs}.
\end{rmk}

\subsection{Equations of flag varieties} \label{sec:eq}
The flag variety \(\Si=G/P \into \PP V_{\lam}\) is defined by an ideal  \(I =<Q>\)  
of quadratic equations generating a linear subspace \(Q \subset Z=S^2 V^*_{\lam}  \)
of the second symmetric power of the contragradient representation \(V^*_{\lam}\). 
The $G$-representation $Z$ has a decomposition 
\begin{displaymath}
Z=V_{2\nu}\oplus V_1 \oplus\cdots \oplus V_n
\end{displaymath}  
into irreducible direct summands, with \(\nu\) being the highest weight of the 
representation \(V^*_{\lam}\). As discussed in \cite[2.1]{rudakov}, 
the subspace \(Q\) in fact consists of all the summands except~\(V_{2\nu}\). 

We compute these equations by using GAP4 code which follows an algorithm given in~\cite{degraaf}. 
The ideal \(I\) is constructed as a certain left ideal of the Verma module corresponding to 
the highest weight \(\lam \) of the Lie algebra \(\gog\). In practice,  
fix a Lie algebra \(\gog\) and dominant weight \(\lam\). Take the left action of 
\(\gog\) on its  universal enveloping algebra \(U(\gog)\) and construct the Verma module 
\(M(\lam)=U(\gog)/N(\lam)\), where \(N(\lam)\) is a certain \(\gog\)-submodule of \(U(\gog)\). 
The module \(M(\lam)\) is also an algebra, and there exist a certain left ideal 
\(I(\lam)\) of it such that \(M(\lam)/I(\lam)\) is an irreducible module of highest 
weight \(\lam\). One can construct the last quotient by calculating 
a Gr{\"o}bner basis of \(I(\lam)\).

The defining equations of the flag varieties appearing in \cite[p.4 and p.20]{wg} can easily 
be recovered using this algorithm, implemented in the GAP4~\cite{GAP4} code given in the Appendix. 

\subsection{Constructing polarized varieties}
We recall the different steps in the construction of polarized varieties as
quasi-linear sections of weighted flag varieties. We will concentrate on the case
of Calabi--Yau threefolds; obvious modifications apply for different classes of 
varieties (Fano, general type, etc). 

\begin{enumerate}\item \textbf{Choose embedding.} We choose a reductive Lie group 
\(G\) and a \(G\)-representation \(V_{\lam}\) of dimension $n$ 
with highest weight \(\lam\). 
We get a straight flag variety $\Si = G/P_\lam \into \PP V_{\lam}$ of 
computable dimension $d$ and codimension~$c=n-1-d$. We choose \(\mu \in \Lam_W^*\) 
and \(u \in \ZZ\) to get an embedding \(w\Si \into w\PP V_{\lam}[w_i]\).
The equations of this embedding can be found as described above.

\item \textbf{Compute Hilbert series and \(K_{w\Si}\).}  We compute the Hilbert series 
of \(w\Si\) by expanding and simplifying (\ref{whhs}) for the given values of 
\(\lam,\mu,u\). If \(w\Si\) is well-formed, we can derive the canonical 
divisor class \(K_{w\Si}\) of \(w\Si.\)

\item \textbf{Take threefold Calabi--Yau section of \(w\Si\).} 
We take a quasi-linear complete intersection
\begin{displaymath}
X= w\Si \cap (w_1)\cap \cdots \cap (w_k)
\end{displaymath} 
inside \(w\Si \) of generic hypersurfaces of degrees equal to some of the weights $w_i$. 
We choose values so that \(K_X=0\) and \(\dim(X)=d-k=3.\) 
This gives the embedding \(X \into \PP[w_0,\cdots,w_s]\), with \(s=n-k.\)
More generally, as in~\cite{wg}, we can take complete intersections inside projective cones
over $w\Si$, adding weight one variables to the coordinate ring which are not involved in any 
relation.

\item \textbf{Check singularities.} As we are interested in quasi-smooth 
Calabi--Yau threefolds, all the singularities of \(X\) must be induced by the weights 
of \(\PP^s[w_i]\). Singular strata of \(\PP^s[w_i]\) correspond to sets of weights 
\(w_{i_0},\cdots,w_{i_p}\) with \(\gcd(w_{i_0},\cdots,w_{i_p})=r\)  
non-trivial, defining a singular stratum $S\subset w\PP$. If the intersection
\(Y=X\cap S\) is non-empty, it has to be a point or a curve, and we need to
find local coordinates in neighbourhoods of points  \(P\in Y\) to check the local
transversal structure. 

\item \textbf{Find invariants and check consistency.} Using the orbifold 
Riemann--Roch formula of~\cite[Section 3]{anita}, we can compute the invariants
of the singular points in our family from the first few values of $h^0(nD)$, and verify 
that the same Hilbert series can be recovered.
\end{enumerate}

\section{Weighted \(G_2\) flag varieties}\label{g2sec}

\subsection{Generalities} Consider the simple Lie group of type $G_2$. Denote by $\al_1, \al_2 \in \Lam_W$ 
a pair of simple roots of the root system $\nabla$ of $G_2$, 
taking $\al_1$ to be the short simple root and $\al_2$ the long one; see Figure \ref{rootsg2}.
The fundamental weights are $ \omega_1=2\al_1+\al_2 $ and $ \omega_2=3\al_1+2\al_2$. 
The sum of the fundamental weights, which is equal to half the sum of the positive roots, 
is $\rho=5\al_1+3\al_2 $. We partition the set of roots into long and short roots as 
$\nabla=\nabla_l \cup \nabla_s \subset \Lam_W $. 
Let ${\{\be_1,\be_2\}}$ be the basis of  the lattice $ \Lam_W^*$ 
dual to \(\{\al_1 , \al_2\}\).

\subsection{The \(G_2\) hypersurface} As a warmup exercise, 
consider the highest weight  $\lam_1=\omega_1=2\al_1+\al_2 $.
\begin{figure}
\centering
\scalebox {1}{ \input{rootsg2.pstex_t}}  
\caption{Root System of \(G_{2}\)}
\label{rootsg2} 
\end{figure}
The irreducible $G_{2}$-representation $V_{\lam_1} $ is seven-dimensional 
\cite[Chapter 22]{harris}, and its projectivisation contains the five-dimensional 
flag variety \(\Si \into \PP^6\) as a hypersurface. The weighted version is 
obtained by letting $ \mu=a\be_1+b\be_2 \in \Lam_W^*$ and $ u \in \ZZ $, leading to the weighted homogeneous variety $$ w\Sigma= w\Sigma(\mu,u) \subset w\PP V_{\lambda_1}.$$

\begin{prop} The Hilbert series of weighted \(G_2\) flag variety embedded in the 
\(G_{2}\)-representation with highest weight \(\om_1=\lam_{1}=2\al_1+\al_2\) is 
\begin{equation} P_{w\Sigma}(t)= \dfrac{(1-t^{2u})} {(1-t^u)\prod_{\alpha \in \nabla_s}\left( 1-t^{<\alpha,\mu>+u}\right) } .
\label{g2hs1}
\end{equation}
Moreover, if \(w\Si\) is well-formed, then \(K_{w\Si}=\Oh(-5u).\)
\end{prop}
\begin{proof}
The set of weights on \(w\Si \subset w\PP V_{\lam_1 }\) is given by \(<w\lam_1,\mu>+u  \) for elements \(w\) of the Weyl group. Since \(\lam_1\) is a  short root, the images of \(\lam_1\) under the action of the Weyl group \(W\) are exactly all the short roots appearing in Figure \ref{rootsg2}. We expand the formula  (\ref{whhs}) for \(\lam_1=\om_1\) and \(W=D_{12}\), 
the dihedral group on 6 letters, and we get 
$$ P_{w\Sigma}(t)= \dfrac{(1+t^u)} {\prod_{\alpha \in \nabla_s}\left( 1-t^{<\alpha,\mu>+u}\right) }.$$
The set of weights on \(w\PP V_{\lam_1}\) is 
\begin{center}
$\lbrace  u, <\alpha, \mu>+u $ for all roots $ \alpha \in \nabla_s\rbrace$. 
\end{center}
We multiply and divide by \((1-t^u)\) to get the full expression for  the Hilbert series, so our denominator contains all weights on \(w\PP V_{\lam_1}\). This gives the Hilbert series in the form of (\ref{g2hs1}) which is in the general form~\eqref{reducedhs}. The sum of the weights on \(w\PP V_{\lam_1}\) is \(7u\), therefore \(K_{w\Si}=\Oh(-5u).\)
\end{proof}

Thus, as is well known, the straight homogeneous variety with $\mu=(a,b)=(0,0)$ and \(u=1\)
is a quadric hypersurface in $\PP^6$:
$$ P_{\Sigma}(t)=\dfrac{1-t^{2}}{(1-t)^7}.
$$

\subsection{The codimension eight weighted \(G_2\)-variety} 
We consider the  $G_2$-representation with highest weight  $\lambda_2=\om_2=3\al_1+2\al_2$. 
The dimension of $ V_{\lam_2} $ is 14 \cite[Chapter 22]{harris}. The homogeneous variety 
$\Sigma \subset \PP V_{\lam_2} $ is five dimensional, so we have an embedding 
\(\Si^5 \into \PP^{13}\) of codimension 8. To work out the weighted version in this case, 
take $\mu=a\be_1+b\be_2 \in \Lam_W^*$ and \(u\in \ZZ\).

\begin{prop} The Hilbert series of the codimension eight weighted \(G_2\) flag variety is given by \begin{equation}  P_{w\Sigma}(t)= \dfrac{ 1- \left( 4+ 2 \sum_{\alpha \in \nabla_s}t^{<\alpha,\mu>} + \sum_{\alpha \in \nabla_s}t^{2< \alpha,\mu >}+\sum_{\alpha \in \nabla_l}t^{< \alpha,\mu >}\right) t^{2u}+\cdots+t^{11u} } {(1-t^u)^2 \prod_{\alpha \in \nabla}\left( 1-t^{<\alpha,\mu>+u}\right) }.\end{equation}  Moreover, if \(w\Si\) is well-formed, then the canonical bundle is $K_{w\Si}=\Oh(-3u)$. \label{hsg2long} \end{prop}
\begin{proof}The set of weights on \(G_2\)-representation with highest weight \(\lam_2=\om_2\) consists of all the roots given in Figure \ref{rootsg2}. The zero weight space appears with multiplicity two. Therefore the set of weights on \(w\PP V_{\lam_2}\) is given by \begin{center}\(\{u,u,<\al,\mu>+u\text{ for all }  \al \in \nabla \text{ of } G_2 \}\).\end{center}
We compute the Hilbert series expression (\ref{whhs}) for \(\lam_2=\om_2\) and Weyl group \(W=D_{12}\).  We get the following form of the Hilbert series.
$$ P_{w\Sigma}(t)= \dfrac{(1+t^u)\left( 1+t^u\left(  1+\sum_{\alpha \in \nabla_s}  t^{< \alpha,\mu >} \right)   +t^{2u}\right) } {\prod_{\alpha \in \nabla_l }\left( 1-t^{<\alpha,\mu>+u}\right) }.$$
Since \(\lam_2 \) is a long root, the images of \(\lam_2\) under the action of the  Weyl group are exactly the long roots. In this form we just have positive coefficients in our numerator. The fully multiplied Hilbert series can be obtained by multiplying and diving the last expression  with 
$$ (1-t^u)^2 \prod_{\alpha \in \nabla_s}\left( 1-t^{<\alpha,\mu>+u}\right).$$
So we get the expression (\ref{hsg2long}). The sum of the weights on \(w\PP V_{\lam_2}\) is \(14u\), therefore the canonical bundle is \(K_{w\Si}=\Oh(-3u).\)
\end{proof}

The Hilbert series of the straight flag variety \(\Si\into \PP^{13}\) can be computed to be 
$$P_{\Sigma}(t)= \dfrac{1-28t^2+105t^3-\cdots+105t^8-28t^9-t^{11}}{(1-t)^{14}}.$$
The image is defined by 28 quadratic equations, listed in Appendix A.

\subsection{Examples}

\begin{example} 
The threefold linear section
\begin{displaymath}
V=\Si  \cap H_1 \cap H_2 \subset \PP^{11},
\end{displaymath}
is a Fano threefold of genus 10, anti-canonically polarised by \(-K_V\) with \((-K_V)^3=18.\) This variety was constructed 
by Mukai using the vector bundle method in \cite{mukai}.

In fact, we failed to find any other $\QQ$-Fano which is a quasilinear section of a codimension 8 
weighted \(G_2\) flag variety or a cone over such. 
A list of 487 possible codimension 8 $\QQ$-Fano 3-folds 
with terminal singularities is available on Gavin Brown's graded ring data base page \cite{brown}. 
We checked that none of them matches the format above.
\end{example}

\begin{example}
\label{g2example}
We construct a family of Calabi--Yau threefolds as quasi-linear sections of the
weighted \(G_{2}\) flag variety \(w\Si \into \PP V_{\lam_2}\). 
 \begin{itemize}
\item Input data: $\mu=2\be_1-3\be_2$, $u=4$.
\item Weighted flag variety: $ w\Si \subset \PP^{13}[1^2,2,3^2,4^4,5^2,6,7^2]$,
with weights assigned to the variables $x_i$ in the equations of Appendix A:
\[
 \renewcommand{\arraystretch}{1.5}
 \begin{array}{ccccccccccccccc}
{\rm Variable} & x_1   & x_2   & x_3   & x_4&x_5&x_6&x_7&x_8&x_9&x_{10}&x_{11}&x_{12}&x_{13}&x_{14}\\
{\rm Weight}  &  6 &   1&           3&    5  &  7 &  4 &      2&    7&     5 &    3&    1&    4&    4&    4
 \end{array}
\]
\item Canonical class: $K_{w\Si}= \Oh(-3u)=\Oh(-12) $, since $ w\Si $ is well-formed.
\item Hilbert numerator: $ 1-t^4-2 t^5-4 t^6-2 t^7-t^8+\cdots+t^{44}$.
\end{itemize}

We take a 3-fold quasilinear section \[ X=w\Si \cap \{x_{4}=f_5\} \cap \{x_8=f_7\} \subset \PP^{11}[1^2,2,3^2,4^4,5,6,7] ,\] 
where $ f_5 $  and  $f_7 $ are general homogeneous forms of degree 5, and 7 respectively 
in the remaining variables, so that 
\[ K_X= \Oh(-12+(7+ 5)) = 0.\]
One can verify that $X$ is well-formed as it does not contain any codimension 9 singular stratum. 
We now check the singularities of $X$ arising from the weights of $w\PP$.

\textbf{$1/7$ singularities}: This  stratum is a single point in $w\PP$ and it is readily seen 
that \(X\) contains this singular point. We  use the implicit function theorem to work out local variables near this point. We can eliminate the variables 
$x_{11},x_{10},x_{12},x_{13},x_{7},x_{2},x_{9},$ and  $x_{3} $ from equations (A.1,6,9,11,12,19,26,28) of Appendix A, respectively. Therefore we get $x_1,x_6$ and $x_{14} $ as local variables near this point on an orbifold cover, the group $\mu_7$ acting by \[ \ep: x_1,x_6,x_{14} \mapsto \ep^6 x_1, \ep^4 x_4, \ep^4 x_{14}. \] So this is an isolated singular point in $X$ of type $ \dfrac{1}{7}(6,4,4)$.

\textbf{$1/6$ singularities}: This stratum is defined by \[ X \cap \lbrace{x_2=x_3=x_5=x_6=x_7=x_9=x_{10}=x_{11}=x_{12}=x_{13}= x_{14}=0 \rbrace} .\] This gives $  x_1^2=0 ,$ which has only a trivial solution. So $X$ does not contain this singular point.

\textbf{$1/5$ singularities}: We have only one variable $ x_9 $ of weight 5 and some pure power of $x_9$ appears in the defining equations of $X$, so it does not contain $1/5$ singular points.

\textbf{$1/4$ singularities}: $X$ restricted to this stratum is defined by \[ \lbrace{x_1=x_2=x_3=x_5=x_7=x_9=x_{10}=x_{11}=0 \rbrace}. \]  We get a rational curve \[ C=\{x_6x_{12}+x_{14}^2=0 \}\subset \PP^2[x_6,x_{12},x_{14}]\] of singularities.
%On each point of the curve one of the $x_i \neq 0. $ First we consider the case $ x_6 \neq 0 .$  By using the implicit function theorem  we get $ x_2, x_3  $ as local variables and the group $ \mu_4 $ acts by \[ \ep : x_2,x_3 \mapsto \ep x_2, \ep^3 x_3. \] For $ x_{12} \neq 0, $  we have  $ x_{10},x_{11} $ as local variables, so the group $ \mu_4 $ acts by \[ \ep : x_{10},x_{11} \mapsto \ep^3 x_{10},\ep x_{11}. \] Similarly for $ x_{14} \neq 0 $, we can check the local variables as $ x_3,x_{11}. $ So we get a rational curve of singularities of type $ \dfrac{1}{4}(1,3).$
Standard calculations show that it is a curve of type $\dfrac{1}{4}(1,3)$.

\textbf{$1/3$ singularities}: On this stratum, the equations again have only the trivial solution, so $X$ does not contain $1/3$ singularities.

\textbf{$1/2$ singularities}: On this stratum $X$ does not contain any new singularities, apart from the curve $C$ above. 

\textbf{Quasi-smoothness}: We have already shown that on singular strata \(X\) is locally a threefold 
by finding the local variables by using the implicit function theorem. One can show by similar 
calculations that on the rest of the strata, \(X\) is locally a smooth threefold. 
Therefore \((X,D)\) is quasi-smooth.

In summary, $(X,D)$ belongs to a family of a polarised Calabi--Yau 3-folds with an isolated singular point of type $\dfrac{1}{7}(6,4,4)$ and a rational curve of singularities $C$ of type $\dfrac{1}{4}(1,3).$ By using the Riemann--Roch formula for Calabi--Yau threefolds 
(\cite[Section 3] {anita}), we can compute the rest of the invariants of this family. 
We get the following data.
\begin {itemize}
 \item $D^3= \dfrac{45}{56}$, \,
  $D.c_2(X)=\dfrac{150}{7}$.
\item  $\deg D\rvert_{C}= \dfrac{1}{2}$, $N_C= -4$ and $\tau_C=1$.
\end {itemize}
One can  recover the Hilbert series of this variety following \cite{anita}, by using these singularities as input data.
\end{example}

\section{Weighted Gr(2,6) varieties}\label{g26sec}
\subsection{The weighted flag variety} 
We take $G$ to be the reductive Lie group of type $\GL(6,\CC)$ with maximal torus $T$. 
In the weight lattice $\Lam_W=<e_1,e_2,e_3,e_4,e_5,e_6>$,
the simple roots are $\left\lbrace \al_i=e_i-e_{i+1} \right\rbrace $, for $1 \leq i \leq 5$.
The Weyl group of \(G\) is \(S_6\) of order 720. The Weyl vector can be taken to be
\begin{displaymath}\rho=5e_1+4e_2+3e_3+2e_4+e_5.\end{displaymath}
Consider the irreducible $G$-representation $V_{\lam}$, with $\lam=e_1+e_2$. Then $V_{\lam}$ is 
15-dimensional, and all of the weights appear with multiplicity one. 
The highest weight orbit space $\Si=G/P_{\lam} \subset \PP V_{\lam}=\PP^{14}$ is eight dimensional.  
This flag variety can be identified with the Grassmannian 
of 2-planes in a 6-dimensional vector space, a codimension 6 variety
\[ 
\Si^8= \Gr(2,6) \hookrightarrow \PP V_{\lam}= \PP^{14}.
 \]

Let $\left\lbrace f_i,  1\leq i \leq 6 \right\rbrace $ be the dual basis of the dual lattice $\Lam_W^*$. 
We choose \[  \mu=  \sum_{i=1}^6a_if_i \in \Lam_W^*,\] and $u \in \ZZ$, 
to get the weighted version of $\Gr(2,6)$,
\[ w\Si(\mu,u)=w\Gr(2,6)_{(\mu,u)} \into w\PP^{14}. \]  
The set of weights on our projective space is $ \left\lbrace <w_i,\mu> + u \right\rbrace  $, where $ w_i $ are weights appearing in the \(G\)-representation \(V_{\lam}\). As a convention we will write an element of dual lattice as row vector, i.e. $\mu=(a_{1},a_2,\cdots,a_6).$

We expand the formula (\ref{whhs}) for the given values of \(\lam,\mu\) to get the following formula for the Hilbert series of \(w\Gr(2,6).\) 
 \[P_{w\Gr(2,6)}(t)=\dfrac{1-Q_1(t)t^{2u}+Q_2(t)t^{3u}-Q_3(t)t^{4u}-Q_4(t)t^{5u}+Q_5(t)t^{6u}-Q_6(t)t^{7u}+
 t^{3s+9u}}{\prod_{1\leq i<j\leq 6} (1-t^{a_i+a_j+u})}.\]
Here 
\[
Q_1(t)=\sum_{1\leq i <j\leq 6}t^{s-(a_i+a_j)}
,\ \ \ Q_2(t)=\sum_{1\leq (i,j)\leq 6}t^{s+(a_i-a_j)}-t^s,\]\[Q_3(t)=\sum_{1\leq i \leq j\leq 6}t^{s+(a_i+a_j)},\ \ \ Q_4(t)=\sum_{1\leq i \leq j\leq 6}t^{2s-(a_i+a_j)},\]
\[Q_5(t)=\sum_{1\leq (i , j)\leq 6}t^{2s+(a_i-a_j)}-t^{2s},\ \ \ Q_6(t)=\sum_{1\leq i \leq j\leq 6}t^{2s+(a_i+a_j)}.\]
In particular, if \(w\Gr(2,6) \into \PP^{14}[<w_i,\mu>+u]\) is well-formed, then its canonical bundle is \(K_{w\Gr(2,6)}=\Oh(-2s-6u)\), with \(s=\sum_{i=1}^6a_i.\)

The defining equations for $ \Gr(2,6)\subset\PP^{14}$ are well known to be the $4 \times 4$ Pfaffians, 
obtained by deleting two rows and the corresponding columns of the $6\times 6$ skew symmetric matrix
\begin{equation} A=
\begin{bmatrix}
0 & x_{1} &x_{2}  &x_{3}  &x_{4}&x_5   \\
   & 0         & x_{6} & x_{7} & x_{8}&x_{9}   \\
    &        & 0          & x_{10} & x_{11}&x_{12}   \\
     &       &             & 0          & x_{13}&x_{14}\\
      &&&&                                0  &x_{15}\\
&&&&&                                          0                          
\end{bmatrix} .
\end{equation}

\subsection{Examples}

\begin{example}\label{seriesgr26} For \(u=1\) and \(\mu=(\underline{0})\), we get the Hilbert series of \(\Gr(2,6)\).$$ P_{\Gr(2,6)}(t)= \dfrac{1-15 t^2+35 t^3-21 t^4-21 t^5+35 t^6-15 t^7+t^9}{(1-t)^{15}}.$$ 
We have $ K_{\Gr(2,6)}= \Oh(-6) $. The three fold hyperplane section $$ V= \Sigma \cap H_1 \cap H_2\cap H_3\cap H_4\cap H_5  \subset \PP^{9}, $$  is a Fano 3-fold with $(-K_V)^3=14$ and genus $g=8$. This variety appears in \cite{mukai}.  

We searched  for more examples of \(\QQ-\)Fano threefolds as quasi-linear sections of 
weighted $\Gr(2,6)$, but as before, didn't succeed. 
\end{example}

\begin{example}
Consider the following data.
\label{expgr261}
\begin{itemize}
\item Input: $ \mu=(2,1,0,0,-1,-2),   u=4 $
\item Variety and Weights: $ w \Gr(2,6) \subset \PP^{14}[1,2^2,3^3,4^3,5^3,6^2,7]  $
\item Canonical class: $ K_{w\Gr(2,6)}= \Oh(-24) $
\item Hilbert Numerator: $ 1-t^5-2 t^6-3 t^7-2 t^8-t^9+\cdots -3 t^{29}-2 t^{30}-t^{31}+t^{36} $.
\end{itemize} 
Consider the threefold  quasilinear section \[ X=w\Gr(2,6)\cap (6)\cap (5)^2\cap (4)^2 \subset \PP^9[1,2^2,3^3,4,5,6,7] ,\] then \[ K_X= \Oh(24-(6+2\times 5+2 \times 4) \sim 0.\]
As in Example~\ref{g2example}, it is easy to check that \((X,D)\) is well-formed and quasi-smooth.
The singularities of \((X,D)\) can also be worked out as before. The conclusion is that 
$(X,D)$ is a Calabi--Yau threefold with an isolated singular point $\dfrac{1}{7}(6,5,3)$ and 
a dissident singular point $\dfrac{1}{6}(1,2,3)$ which lies on the intersection of two rational 
curves of singularities $C$ and $E$ of types $\dfrac{1}{3}(1,2)$ and $\dfrac{1}{2}(1,1)$ respectively. 
The rest of data of our variety and its singularities is as follows.
\begin {itemize}
 \item $D^3= \dfrac{11}{21}$, $D.c_2(X)=16$.
\item  $\deg D\rvert_{C}= \dfrac{1}{2}$, $N_C= 3$ and $\tau_C=2$.
\item $\deg D\rvert_{E}= \dfrac{2}{3}$, $N_E= 1$ and $\tau_E=1$.
\end {itemize}
\end{example}

\newpage \appendix \label{g2eq}

\section{Equations of the $G_2$ flag variety}
The codimension eight $G_2$ flag variety is defined by the following 28 quadrics. 
\begin{eqnarray}
&& \textstyle x_{14}^2+x_{13}x_{14}+\frac{1}{3}x_{13}^2+x_{6}x_{12}+\frac{1}{3}x_{4}x_{10}+x_{5}x_{11}+\frac{1}{3}x_{3}x_{9}+x_{2}x_{8}+\frac{1}{3}x_{1}x_{7}\\
&& \textstyle x_{14}^2+2x_{13}x_{14}+\frac{8}{9}x_{13}^2+x_{6}x_{12}+3x_{5}x_{11}-\frac{4}{9}x_{4}x_{10}-\frac{7}{9}x_{3}x_{9}-\frac{5}{9}x_{1}x_{7}\\
&& \textstyle 2x_{13}x_{14}+\frac{5}{3}x_{13}^2+9x_{5}x_{11}-\frac{5}{3}x_{4}x_{10}-\frac{7}{3}x_{3}x_{9}+3x_{2}x_{8}-\frac{4}{3}x_{1}x_{7}\\
&&  \textstyle  x_{13}^2 +9x_{5}x_{11}-3x_{4}x_{10}-3x_{3}x_{9}+9x_{2}x_{8}-2x_{1}x_{7}\\
&&  \textstyle x_{11}x_{14}+\frac{1}{3}x_{11}x_{13}+\frac{1}{9}x_{7}x_{10}+x_{2}x_{12} \\
&&  \textstyle x_{5}x_{10}-\frac{2}{3}x_{4}x_{9}-3x_{3}x_{8}+2x_{1}x_{14}+\frac{1}{3}x_{1}x_{13}\\
&&  \textstyle x_{10}x_{14}+\frac{4}{9}x_{10}x_{13}+\frac{2}{9}x_{7}x_{9}-x_{3}x_{12}+\frac{1}{3}x_{1}x_{11}\\
&&  \textstyle x_{9}x_{14}+x_{9}x_{13}-x_{7}x_{8}+x_{4}x_{12}+\frac{2}{3}x_{1}x_{10}\\
&& \textstyle  x_{8}x_{14}+\frac{2}{3}x_{8}x_{13}-x_{5}x_{12}+\frac{1}{9}x_{1}x_{9}\\
&& \textstyle  2x_{7}x_{14}+\frac{5}{3}x_{7}x_{13}-3x_{4}x_{11}-\frac{2}{3}x_{3}x_{10}+x_{2}x_{9}\\
&& \textstyle  x_{6}x_{8}+x_{5}x_{14}+\frac{1}{3}x_{5}x_{13}+\frac{1}{9}x_{1}x_{4} \\
&& \textstyle  3x_{5}x_{7} -x_{6}x_{9}+x_{4}x_{14}+\frac{4}{3}x_{4}x_{13}-\frac{2}{3}x_{1}x_{3}\\
&& \textstyle   x_{6}x_{10}-\frac{2}{3}x_{4}x_{7}+x_{3}x_{14}-\frac{1}{3}x_{3}x_{13}+3x_{1}x_{2}\\
&& \textstyle  -x_{6}x_{11}+\frac{1}{9}x_{3}x_{7}+x_{2}x_{14}+\frac{2}{3}x_{2}x_{13} \\
&& \textstyle   x_{12}x_{13}-\frac{1}{3}x_{9}x_{10}+3x_{8}x_{11}  \\
&& \textstyle   x_{10}x_{13}-x_{7}x_{9}+3x_{1}x_{11}\\
&& \textstyle   x_{9}x_{13}-3x_{7}x_{8}+x_{1}x_{10}   \\
&& \textstyle  x_{7}x_{13}-3x_{4}x_{11}+3x_{2}x_{9}   \\
&& \textstyle   x_{6}x_{13}-\frac{1}{3}x_{3}x_{4}+3x_{2}x_{5}   \\
&& \textstyle  3x_{5}x_{7}+x_{4}x_{13}-x_{1}x_{3}  \\
&& \textstyle   x_{4}x_{7}+x_{3}x_{13}-3x_{1}x_{2} \\
&& \textstyle  x_{1}x_{13}-3x_{5}x_{10}+3x_{3}x_{8} \\
&& \textstyle   x_{9}x_{11}-\frac{1}{3}x_{10}^2+x_{7}x_{12} \\
&& \textstyle  \frac{1}{3}x_{9}^2-x_{8}x_{10}+x_{1}x_{12}  \\
&& \textstyle   \frac{1}{3}x_{7}^2+x_{3}x_{11}+x_{2}x_{10} \\
&& \textstyle  x_{5}x_{9}+x_{4}x_{8}+\frac{1}{3}x_{1}^2 \\
&& \textstyle   x_{6}x_{7}+\frac{1}{3}x_{3}^2-x_{2}x_{4} \\
&& \textstyle  x_{1}x_{6} -\frac{1}{3}x_{4}^2+x_{3}x_{5}
\end{eqnarray}

This list was obtained using the algorithm described above in Section~\ref{sec:eq}, 
with the help of the GAP4 code given below, kindly provided to us by Willem De Graaf. 
With obvious modifications, this code in principle computes the equations of any flag variety.

%\vspace{0.2in}
\newpage 

{\tt 
hwvecs:= function( L, V )

    local x, n, eqs, i, j, cfs, k, sol;

    x:= ChevalleyBasis( L )[1];
    n:= Dimension( V );
    
    eqs:= NullMat( n, n*Length(x) );
    
    for i in [1..Length(x)] do
    
        for j in [1..n] do
        
         cfs:= Coefficients( Basis(V), $\rm x[i]{}^{\wedge}${\tt Basis(V)[j])};

             {\tt  for k in [1..n] do
                
eqs[j][(i-1)*n+k]:= cfs[k];

            od;
        od;
    od;

    sol:= NullspaceMatDestructive( eqs );

    return List( sol, u -> LinearCombination( Basis(V), u ) );

end;

L:=SimpleLieAlgebra("G",2,Rationals);

V:= AdjointModule( L );

 W:= SymmetricPowerOfAlgebraModule( V, 2 );

 hwv:= hwvecs( L, W );

 M:= List( hwv, u -> LeftAlgebraModuleByGenerators( L, \textbackslash\textasciicircum,\ [u] ) );

 List( M, Dimension );

Basis(M[2]);Basis(M[3]);}}

\noindent {\sc Mathematical Institute, University of Oxford.}

\noindent {\sc 24-29 St Giles', Oxford, OX1 3LB, United Kingdom.}

\vspace{0.1in}

\noindent {\tt qureshi@maths.ox.ac.uk}

\noindent {\tt szendroi@maths.ox.ac.uk}

\end{document}